\documentclass[reqno,10pt,centertags]{amsart}
\usepackage{amsmath,amsthm,amscd,amssymb,latexsym,esint,upref,stmaryrd,
enumerate,color,verbatim,yfonts,mathrsfs} 

\usepackage{hyperref}

\newcommand*{\mailto}[1]{\href{mailto:#1}{\nolinkurl{#1}}}

\newcommand{\cH}{{\mathcal H}}

\newcommand{\cK}{{\mathcal K}}

\DeclareMathOperator{\ran}{ran}
\DeclareMathOperator{\dom}{dom}

\newcommand{\no}{\notag}
\newcommand{\lb}{\label}

\newcommand{\ol}{\overline}

\newcommand{\bi}{\bibitem}

\let\geq\geqslant

\makeatletter
\def\theequation{\@arabic\c@equation}

\allowdisplaybreaks 
\numberwithin{equation}{section}

\newtheorem{theorem}{Theorem}[section]
\newtheorem{proposition}[theorem]{Proposition}

\theoremstyle{remark}
\newtheorem{remark}[theorem]{Remark}

\begin{document}

\numberwithin{equation}{section}
\allowdisplaybreaks

\title[Some Remarks on the Operator $T^*T$]{Some Remarks on the Operator $T^*T$} 
    
\author[F.\ Gesztesy]{Fritz Gesztesy}
\address{Department of Mathematics, 
Baylor University, One Bear Place \#97328,
Waco, TX 76798-7328, USA}
\email{\mailto{Fritz\_Gesztesy@baylor.edu}}
\urladdr{\url{http://www.baylor.edu/math/index.php?id=935340}}

\author[K.\ Schm\"udgen]{Konrad Schm\"udgen}
\address{Department of Mathematics,
University of Leipzig Leipzig, Germany}
\email{\mailto{Konrad.Schmuedgen@math.uni-leipzig.de}}
\urladdr{\url{http://www.math.uni-leipzig.de/~schmuedgen/}}



\date{\today}

\begin{abstract} 
This note deals with the operator $T^*T$, where $T$ is a densely defined operator 
on a complex Hilbert space. We reprove a recent result of Z.\ Sebesty\'en and Zs.\ Tarcsay \cite{ST14}: If $T^*T$ and $TT^*$ are self-adjoint, then $T$ is closed. In addition, we describe the Friedrichs extension of $S^2$, where $S$ is a symmetric operator, recovering results due to Yu. Arlinski{\u i} 
and Yu.\ Kovalev \cite{AK11}, \cite{AK13}.
\end{abstract}

\maketitle 


\section{Introduction}  \lb{s1}

Let $T$ be a densely defined closed linear operator on a Hilbert space $\cH_1$ into another Hilbert space $\cH_2$. The operator $T^*T$ is an important technical tool in  operator theory. Basic facts  such as, $T^*T$ is a nonnegative self-adjoint  operator in $\cH_1$, and $T^*T + I_{\cH_1}$ is a bijection of 
$\cH_1$ onto itself, are usually proved at the beginning of treatments of unbounded operators on Hilbert spaces. 

In the present note we discuss two topics concerning the operator $T^*T$.  In Section \ref{s2} we  provide an alternative short proof of a converse to a theorem of von Neumann 
\cite[Satz~3]{Ne32} obtained by Z.\ Sebesty\'en and Zs.\ Tarcsay in 2014 \cite[Theorem~2.1]{ST14}. In Section \ref{s3} we consider a symmetric operator $S$ on a Hilbert space such that  $\dom \big(S^2\big)$ is dense. We  determine the Friedrichs extension $\big(S^2\big)_F$ of the nonnegative symmetric operator $S^2$ and discuss when this Friedrichs extension equals the operator $S^*\ol{S}$, recovering a result due to Yu. Arlinski{\u i} and Yu.\ Kovalev \cite{AK11}, ]\cite{AK13}.

Throughout this note, the symbols $\cH$, $\cH_1$, and $\cH_2$ denote complex Hilbert spaces and 
$(\, \cdot \,, \, \cdot \,)_\cH$ denotes the scalar product of $\cH$. By a {\it Hermitian} operator in this note we mean a (not necessarily densely defined) linear operator $T$ on $\cH$ such that 
$(Tf,g)_{\cH}=(f,Tg)_{\cH}$ for all $f,g \in \dom (T)$. In the case where $S$ is Hermitian and densely defined, we call $S$ {\it symmetric}. Clearly, $S$ is symmetric if and only if $S \subseteq S^*$.

\section{On a Converse to a Theorem of von Neumann Due to \\ Z.\ Sebesty\'en and Zs.\ Tarcsay} \lb{s2}

In this section we provide a short proof of a converse to a theorem of von Neumann 
\cite[Satz~3]{Ne32} obtained by Z.\ Sebesty\'en and Zs.\ Tarcsay in 2014 \cite[Theorem~2.1]{ST14}.

Von Neumann's theorem \cite[Satz~3]{Ne32}, in a two-Hilbert space formulation as in 
\cite[Theorem~V.3.24]{Ka80} (see also \cite[p.~1245--1247]{DS88}, 
\cite[Theorem~3.1]{EE89}, \cite[Theorem~X.25]{RS75}, 
\cite[p.~312--313]{RS90}, \cite[Proposition~3.18]{Sc12}, \cite[Theorem~5.39]{We80}) reads as follows:

\begin{theorem} \lb{t2.1}
Let $T$ be a closed and densely defined linear operator of $\cH_1$ into $\cH_2$. Then $T^*T$ is a nonnegative and self-adjoint operator in $\cH_1$. 
\end{theorem}

It is well-known (cf.\ \cite[Theorem~V.3.24]{Ka80} or \cite[Example~VI.2.13]{Ka80}) that under the hypotheses of Theorem \ref{t2.1} one also has that $\dom(T^*T)$ is a core for $T$.

The converse of von Neumann's theorem obtained by Z.\ Sebesty\'en and Zs.\ Tarcsay  
\cite[Theorem~2.1]{ST14} in 2014 is of the form:

\begin{theorem} \lb{t2.2}
Let  $T$ be a  densely defined  linear operator of $\cH_1$ into $\cH_2$. In addition, suppose that $T^*T$ and $TT^*$ are self-adjoint in $\cH_1$ and $\cH_2$, respectively. Then $T$ is closed. 
\end{theorem}

Combining Theorems \ref{t2.1} and \ref{t2.2} then yields the following fact (cf.\ 
\cite[Theorem~3.2]{ST14}):

\begin{theorem} \lb{t2.3}
Let $T$ be a densely defined linear operator of $\cH_1$ into $\cH_2$. Then the following assertions $(i)$ and $(ii)$ are equivalent: \\[1mm]
$(i)$ $T$ is closed. \\[1mm]
$(ii)$ $T^*T$ and $TT^*$ are $($nonnegative\,$)$ self-adjoint operators in $\cH_1$ and $\cH_2$, respectively. 
\end{theorem}

Next, we present a short proof of Theorem \ref{t2.2} on the basis of an elegant unpublished approach due to Ed Nelson. (For applications of this approach see, for instance, \cite{EGNT14}, \cite{Th88}, 
\cite[Chs.~5, 6]{Th92}, and the references cited therein.)

\begin{proof}[Proof of Theorem~\ref{t2.2}]
Introduce in $\cH_1 \oplus \cH_2$ the $2 \times 2$ block matrix operator
\begin{equation}
Q = \begin{pmatrix} 0 & T^* \\ T & 0 \end{pmatrix}, \quad \dom(Q) = \dom(T) \oplus \dom(T^*). 
\lb{2.1} 
\end{equation}
then (cf.\ also \cite[Proposition~2.6.3]{Tr08}),
\begin{equation}
Q^* = \begin{pmatrix} 0 & T^* \\ \ol{T} & 0 \end{pmatrix}, 
\quad \dom(Q^*) = \dom\big(\ol{T}\big) \oplus \dom(T^*),      \lb{2.2}
\end{equation}
and hence $Q$ is symmetric in $\cH_1 \oplus \cH_2$, and 
\begin{equation}
\ol{Q} = (Q^*)^* = \begin{pmatrix} 0 & \big(\ol{T}\big)^* \\ (T^*)^* & 0 \end{pmatrix} 
= \begin{pmatrix} 0 & T^* \\ \ol{T} & 0 \end{pmatrix} , 
\quad \dom\big(\ol{Q}\big) = \dom\big(\ol{T}\big) \oplus \dom(T^*).     \lb{2.3}
\end{equation} 
By hypothesis,
\begin{equation}
Q^2 = \begin{pmatrix} T^*T & 0 \\ 0 & TT^* \end{pmatrix} \geq 0, 
\quad \dom\big(Q^2\big) = \dom(T^*T) \oplus \dom(TT^*),       \lb{2.4}
\end{equation}
is self-adjoint in $\cH_1 \oplus \cH_2$. Hence, 
\begin{align}
\cH_1 \oplus \cH_2 &= \ran\big(Q^2 + I_{\cH_1} \oplus I_{\cH_2}\big)   \no \\
&= \ran((Q - iI_{\cH_1} \oplus I_{\cH_2})(Q + iI_{\cH_1} \oplus I_{\cH_2}))    \lb{2.5} \\
&= \ran((Q + iI_{\cH_1} \oplus I_{\cH_2})(Q - iI_{\cH_1} \oplus I_{\cH_2})),   \no 
\end{align}
and one concludes that 
\begin{equation}
\ran((Q \pm iI_{\cH_1} \oplus I_{\cH_2}) = \cH_1 \oplus \cH_2,     \lb{2.6} 
\end{equation}
and thus $Q$ is self-adjoint (and hence closed), implying closedness of $T$ by \eqref{2.1} and \eqref{2.3}. 
\end{proof}

Since $\dom\big(Q^2\big)$ is clearly a core for $Q$, this proof also yields the fact that 
$\dom(T^*T)$ and $\dom(TT^*)$ are cores for $T$ and $T^*$, respectively.

We also note that the factorization argument in \eqref{2.5} extends to higher integer powers than $2$ as discussed in \cite[Theorem~5.22]{We80}.  

For further results on operators of the type $T^*T$ when $T$ is not necessarily assumed to be 
closed, we refer to \cite{ST12}.

\section{On the Friedrichs Extension of $S^2$, with $S$ Symmetric} \lb{s3}

In this section  we reconsider  the Friedrichs extension of the nonnegative operator $S^2$, where $S$ is a symmetric operator on the complex Hilbert space $\cH$. 

After submitting our note to the archive, we were informed by Yu.\ Arlinski{\u i} that the results presented in this section are contained in two papers by him and Yu.\ Kovalev \cite[Theorem~1.1, Proposition~3.3]{AK11}, \cite[Theorems~1.1, 1.2]{AK13}. We decided to keep this section in this archive submission as the approach in Theorem \ref{t3.1} below differs a bit from that in \cite{AK11} and \cite{AK13}. However, we emphasize that papers \cite{AK11} and \cite{AK13} go far beyond the scope of this section, in particular, they also discuss the Krein--von Neumann extension in addition to the Friedrichs extension.

To begin our discussion we recall a special case of the 1st form representation theorem (cf., e.g., 
\cite[Theorem~IV2.4]{EE89}, \cite[Theorem~VI.2.1]{Ka80}, \cite[Theorems~VIII.15, VIII.16]{RS80}, 
\cite[Theorem~10.7]{Sc12}, \cite[Theorem~7.5.5, Proposition~7.5.6]{Si15}, \cite[Theorem~5.36]{We80} 
for the latter) in the following form: Let $T$ be a densely defined operator in $\cH$ and introduce the sesquilinear form $q_T$ in $\cH$ (actually, $\cH \times \cH$, but we will abuse notation a bit) via
\begin{equation}
q_T(f,g) = (Tf, Tg)_{\cH}, \quad f, g \in \dom(q_T) = \dom(T).    \lb{3.1}
\end{equation}
Then $q_T$ is closable (resp., closed) if and only if $T$ is closable (resp., closed). (See, e.g., 
\cite[Example~VI.2.13]{Ka80}). Moreover, if $T$ is closable, then the closure of $q_T$, denoted by 
$\ol{q_T}$, is then given by $q_{\ol T}$, that is, 
\begin{equation}
\ol{q_T}(f,g) = ({\ol T}f, {\ol T}g)_{\cH}, \quad f, g \in \dom\big(\ol{q_T}) = \dom\big({\ol T}\big).  \lb{3.2}
\end{equation}
In this case the nonnegative, self-adjoint operator uniquely associated with $\ol{q_T} = q_{\ol T}$ is given by 
\begin{equation}
T^* {\ol T} = \big({\ol T}\big)^* {\ol T} \geq 0.    \lb{3.3}
\end{equation} 

Given these preparations we can state the following result:

\begin{theorem} \lb{t3.1} 
Suppose $S$ is symmetric in $\cH$ such that $\dom\big(S^2\big)$ is dense in $\cH$. \\[1mm] 
$(i)$ The 
Friedrichs extension of $S^2$, denoted by $\big(S^2\big)_F$, is given by 
\begin{equation}
\big(S^2\big)_F = \Big(\ol{S \big|_{\dom(S^2)}}\Big)^* \ol{S \big|_{\dom(S^2)}} \geq 0.    \lb{3.4}
\end{equation}
In general,  $\big(S^2\big)_F$ differs  from $S^* \ol{S}=\big(\ol{S}\big)^* \ol{S}$. \\[1mm] 
$(ii)$ If one of the deficiency indices of $\ol S$ is finite, then 
 $\big(S^2\big)_F=S^* \ol{S}$, or, equivalently, $\dom\big(S^2\big)$ is a core for $S$.  
 \end{theorem}
\begin{proof}
$(i)$ Since $\dom\big(S^2\big)$ is dense in $\cH$, $S \big|_{\dom(S^2)}$ is closable and hence so is the form 
\begin{equation}
Q_{S^2} (f,g) = \big(f, S^2 g\big)_{\cH}, \quad f,g \in \dom(Q_{S^2}) = \dom\big(S^2\big).    \lb{3.5}
\end{equation}
Rewriting $Q_{S^2}$ as
\begin{align}
\begin{split} 
& Q_{S^2} (f,g) = \big(S f, S g\big)_{\cH} = \big(S \big|_{\dom(S^2)} f, S \big|_{\dom(S^2)} g\big)_{\cH},  \\
& f,g \in \dom(Q_{S^2}) = \dom\big(S^2\big),    \lb{3.6}
\end{split} 
\end{align}
an application of \eqref{3.1}, \eqref{3.2} then yields for the closure $\ol{Q_{S^2}}$ of $Q_{S^2}$,
\begin{equation}
\ol{Q_{S^2}} (f,g) = \big(\ol{S \big|_{\dom(S^2)}} f, \ol{S \big|_{\dom(S^2)}} g\big)_{\cH}, 
\quad f,g \in \dom(Q_{S^2}) = \dom\big(S^2\big).     \lb{3.7}
\end{equation} 
By its definition (see, e.g., \cite[p.~325--326]{Ka80}, \cite[p.~177--181]{RS75}, \cite[p.~234--241]{Sc12}) 
the Friedrichs extension of $S^2$ is the nonnegative self-adjoint operator associated with the form $\ol{Q_{S^2}}$\,. Therefore,  \eqref{3.4} follows by combining \eqref{3.7} and  \eqref{3.3}.

Next, set $R:=\ol{S \big|_{\dom(S^2)}}$. Then $(S^2\big)_F=R^*R$ by \eqref{3.4}. If $\dom\big(S^2)$ 
is a core for $S$, then  $R=\ol{S}$ and hence $\big(S^2\big)_F=S^* \ol{S}$. Conversely, suppose that 
$\big(S^2\big)_F\equiv R^*R=S^* \ol{S}$. Then $|R|= \big|\ol{S}\big|$ and hence 
$\dom(R)=\dom(|R|) = \dom \big(\big|\ol{S}\big|\big) = \dom\big(\ol{S}\big)$. Since 
$R \subseteq \ol{S}$, this yields $R=\ol{S}$. Thus, $\big(S^2\big)_F=S^* \ol{S}$\, if and only if  
$\dom\big(S^2\big)$ is a core for $S$.

The existence of a symmetric operator $S$ for which $\dom\big(S^2\big)$ is dense, but not a 
core for $S$, follows from \cite[Corollary 4.3]{Sc83}. \\[1mm] 
\noindent  
$(ii)$ If $S$ is closed and has at least one finite deficiency index, then by Proposition \ref{p3.2}\,$(iii)$ below, $\dom\big(S^2\big)$ is a  core for $S$ and therefore $\big(S^2\big)_F=S^* \ol{S}$.
\end{proof}

Theorem \ref{t3.1}\,$(i)$ shows that the second Corollary in \cite[p.~181]{RS75} needs to be revised accordingly. 

In the remaining part of this section we use the Cayley transform of $T$ to investigate  when 
$\dom\big(T^2\big)$ is a core for $T$. 

Let $T$ be a closed Hermitian operator on $\cH$. Then $(T \pm i I_{\cH})\dom(T)$ are closed linear subspaces of $\cH$. Their orthogonal complements $\cH_\pm:=((T \mp i I_{\cH})\dom (T))^\bot$ are the deficiency subspaces of $T$. Next, we recall that the Cayley transform $V_T$ of $T$ is the partial isometry 
mapping $(T + i I_{\cH})\dom (T)$ to $(T - i I_{\cH})\dom (T)$ via
\begin{equation} 
V_T (T + i I_{\cH})f=(T - i I_{\cH})f,\quad f\in \dom (T).
\end{equation} 
Thus, $\dom (T) = (I_{\cH} - V_T)(I_{\cH} - P_+)\cH,$  where $P_+$ is the orthogonal  projection onto 
$\cH_+$. We extend $V_T$ to an operator on the entire Hilbert space $\cH$ by setting $V_T f = 0$ for 
$f \in \cH_+$. Let $P$ be the orthogonal projection onto the closure of the linear subspace 
$\cH_+ + V_T^*\cH_+$ of $\cH$. (Here $X+Y = \{z = x + y \,|\, x \in X, y \in Y\}$, with $X, Y$ linear subspaces of $\cH$.) We note that in general $\cH_+ + V_T^*\cH_+$ is different from 
$(I_{\cH_+} + V_T^*)\cH_+$.

The following proposition is an adaption of some consideration from \cite{Sc83}. 

\begin{proposition}\label{p3.2} Suppose that $T$ is a closed Hermitian operator on $\cH$. \\[1mm] 
$(i)$ $\dom \big(T^2\big) = (I_{\cH} - V_T)^2(I_{\cH} - P)\cH$. \\[1mm] 
$(ii)$ $\dom \big(T^2\big)$ is a core for $T$ if and only if $(I_{\cH} - V_T^*)h\in P \cH$ for 
$h\in (I_{\cH} - P_+)\cH$ implies that $h=0$. \\[1mm] 
$(iii)$ Suppose, in addition, that $T$ is densely defined in $\cH$. If $\cH_+ +V_T^*\cH_+$ is closed, or, if at least one of the deficiency indices of $T$ is finite, then $\dom \big(T^2\big)$ is a core for $T$.
\end{proposition}
\begin{proof}
$(i)$ Let\, $f\in \dom \big(T^2\big)$. Then\, $f\in \dom (T)$, so that $f=(I_{\cH} - V_T)g$,\, where\, 
$g\in(I_{\cH} - P_+)\cH$. Furthermore, 
$(T + i I_{\cH})f= 2ig\in \dom (T)$ implies that $g=(I_{\cH} - V_T)h $ with\,  $h \in(I_{\cH} - P_+) \cH$. 
Since $h,g \perp \cH_+$, one obtains $g = (I_{\cH} - V_T)h\perp \cH_+$. Hence we conclude that 
$h \perp \cH_+$ and $h \perp V_T^* \cH_+$, that is, $h \perp P \cH$. Thus, 
$f=(I_{\cH} - V_T)g = (I_{\cH} - V_T)^2 h \in(I_{\cH} - V_T)^2(I_{\cH} - P)\cH$. Thus, 
$\dom \big(T^2\big) \subseteq (I_{\cH} - V_T)^2(I_{\cH} - P)\cH.$
Reversing this reasoning one obtains the converse inclusion. \\[1mm]  
\noindent 
$(ii)$ Since $T$ is symmetric,  one has $\|(T + i I_{\cH})f\|_{\cH}^2 = \|Tf\|_{\cH}^2+\|f\|_{\cH}^2$ for 
$f\in \dom (T)$. Since $T$ is closed, 
$\cK:=(\dom (T), (\, \cdot \,, \, \cdot \,)_{\cK} := ((T + i I_{\cH}) \, \cdot \,,(T + i I_{\cH}) \, \cdot \,)_{\cH})$ is a Hilbert space.  Therefore, $\dom \big(T^2\big)$ is a core for $T$ if and only if 
the orthogonal complement of $\dom \big(T^2\big)$ in $\cK$ is trivial. Fix $g=(I_{\cH} - V_T)h $, where 
$h\in(I_{\cH} - P_+) \cH$,  in $\cK$. By item $(i)$, $\dom \big(T^2\big)$ is the set of vectors 
$(I_{\cH} - V_T)^2f$ with $f\in (I_{\cH} - P)\cH$. Hence, $g\in \big(\dom \big(T^2\big)\big)^\bot$ in $\cK$ 
if and only if for all $f\in (I_{\cH} - P)\cH$ one has
\begin{align}
\begin{split} 
0 &= ((T + i I_{\cH})(I_{\cH} - V_T)^2f, (T + i I_{\cH})g)_{\cH}=4 ((I_{\cH} - V_T)f,h)_{\cH}   \\
&= 4 (f,(I_{\cH} - V_T^*)h)_{\cH},
\end{split} 
\end{align}
or, equivalently,  $(I_{\cH} - V_T^*)h\in P \cH$. Since $g=0$ if and only if $h=0$ (by 
$(T + i I_{\cH}) g=2ih$), one concludes that $\big(\dom \big(T^2\big)\big)^\bot =\{0\}$ in $\cK$ if and only if the condition stated in item $(ii)$ holds. \\[1mm]  
\noindent 
$(iii)$ First we suppose that $\cH_+ +V_T^*\cH_+$ is closed. Assume that $(I_{\cH} - V_T^*)h\in P \cH$ for some $h\in (I_{\cH} - P_+)\cH$. Then can write $(I_{\cH} - V_T^*)h=(I_{\cH} - V_T^*)f_1+f_2$ with 
$f_1,f_2\in \cH_+$. For $g\in \cH$, one has
\begin{align}
\begin{split} 
((h-f_1), (I_{\cH} - V_T)(I-P_+)g)_{\cH} & = ((I_{\cH} - V_T^*)(h-f_1),(I-P_+)g)_{\cH}    \\
&= ( f_2,(I-P_+)g)_{\cH} = 0. 
\end{split} 
\end{align}
This shows that $(h-f_1) \bot \dom (T)$. Since $\dom (T)$ is dense,  $h=f_1\in \cH_+$. Combined with 
$h\in (I_{\cH} - P_+)\cH$, the latter yields $h=0$. Thus, the condition in item $(ii)$ is fulfilled and 
$\dom \big(T^2\big)$ is a core for $T$.

Next, suppose that $T$ has at least one finite deficiency index. Upon replacing $T$ by $-T$, if necessary, one can assume without loss of generality that $\dim (\cH_+) < \infty$. Then $\cH_+ + V_T^*\cH_+$ is 
finite-dimensional and hence closed.
\end{proof}

\begin{remark} \lb{r3.3}
We emphasize once more that the results presented in this section are contained in two papers by 
Yu.\ Arlinski{\u i} and Yu.\ Kovalev \cite[Theorem~1.1, Proposition~3.3]{AK11}, 
\cite[Theorems~1.1, 1.2]{AK13}. These references contain many additional results, including a discussion of the Krein--von Neumann extension. 
\end{remark}


\medskip

\noindent
{\bf Acknowledgments.} We are indebted to Yury Arlinski{\u i} for kindly pointing out to us his papers with Yury Kovalev \cite{AK11}, \cite{AK13}, which contain the results presented in Section \ref{s3}. In addition, we are grateful to Maxim Zinchenko for very helpful comments. 
 
 
\end{document}